\documentclass{siamltex}

\usepackage{amsmath,amssymb}
\usepackage{mathrsfs}
\usepackage{geometry}
\usepackage{tikz,pgfplots}
\usepackage{float}

\newcommand{\la}{\langle}
\newcommand{\ra}{\rangle}

\def\+{{\dagger}}

\def\CC{{\mathbb C}}

\def\NN{{\mathbb N}}

\def\RR{{\mathbb R}}
\def\ZZ{{\mathbb Z}}

\newcommand{\fH}{\mathfrak{H}}

\newcommand{\fL}{\mathfrak{L}}

\def\sL{{\mathscr L}}

\def\sR{{\mathscr R}}

\newtheorem{thm}{Theorem}
\newtheorem{lem}[thm]{Lemma}

\newtheorem{prop}[thm]{Proposition}

\title{Preconditioned Random Toeplitz Operators}
\author{W. F. Ke\footnotemark[2]
\and K. F. Lai\footnotemark[3]\ \footnotemark[5]
\and N. C. Wong\footnotemark[4]}

\begin{document}

\maketitle

\renewcommand{\thefootnote}{\fnsymbol{footnote}}\makeatletter
\footnotetext[2]{Department of Mathematics and National Center for Theoretical Sciences, National Cheng Kung University, Tainan, Taiwan (wfke@mail.ncku.edu.tw)}
\footnotetext[3]{School of Mathematical Sciences, Capital Normal University, Beijing 100048, China
(kinglaihonkon@gmail.com)}
\footnotetext[4]{Department of Applied Mathematics, National Sun Yat-sen University, Kaohsiung, Taiwan
(wong@math.nsysu.edu.tw)}
\footnotetext[5]{We thank the National Center of Theoretical Science for the support for a short visit during which this work is done.}

\renewcommand{\thefootnote}{\arabic{footnote}}

\begin{abstract}
  The solution of Hermitian positive definite random Toeplitz systems $Ax=b$ by the preconditioned
  conjugate gradient method for the Strang circulant preconditioner is studied. We established the foundation
  for this method by extending the work of
  Brown-Halmos on Toeplitz operators and
  Grenander-Szeg\"o on Teoplitz form to random Teoplitz operators.
\end{abstract}



\section{Introduction}

The aim of this paper is to show that we can  precondition a random Toeplitz operator to yield fast convergence  when conjugate gradient method is used for computation.

Preconditioned conjugate gradient (PCG) for solving finite dimensional linear system $Ax=b$ is a well established technique in numerical linear algebra (see for example \cite{Axe 94}, \cite{GV 96}).
When $A$ is a Toeplitz matrix very efficient preconditioners have been found (see \cite{Str 86}, \cite{CJ 07}).
Conjugate gradient methods also works for Hilbert spaces (see \cite{Dan 70},
\cite{Pan 04}).

Since the innovative work of Wigner \cite{Wig 55} random matrices with independent identically distributed (iid) random variables as entries have been intensively studied (see \cite{BS 10}, \cite{AGZ 09}, \cite{Tao 12}, \cite{PS 11}) with many useful applications
(see \cite{Meh 04}, \cite{BCC 09}).
The special class of random Toeplitz matrices, in particular their spectral measure
have been studied by Bryc and others (\cite{BDJ 06}).
 But there has been no work on their eigenvalue distributions in
relation to the properties of a generating function as given in the classic work of Grenander and Szeg\"o \cite{GS 58}.

 On the other hand we can consider random matrices as random linear operators as given in \cite{Sko 84}. This is the point of view we shall take. We shall establish for random Toeplitz operators the theoretical background used by Raymond Chan in his
 important work on circulant preconditioners for Toeplitz matrices, see in particular \cite{ChS 89}, \cite{ChR 89}, \cite{ChR 91}.

 The paper is divided into three parts. In the first part we shall formulate and prove for random matrices some theorems on
the distribution of eigenvalues of random matrices which are  standard
in the case of number matrices. We think it is useful to have these theorems written down and they
 will be used in subsequent papers on applications to numerical computations.
In part two we first extend the work of Brown and Halmos \cite{BH 63} on Toeplitz matrix to random Toeplitz operators.
And then we extend the results of Grenander and Szeg\"o \cite{GS 58} on Toeplitz forms to random Toeplitz operators.
These results are what we need for the extension of Raymond Chan's results to random Toeplitz operators.
In part three we apply the results of part two to show that that we can use the Strang's circulant to precondition a random Toeplitz operator and give some numerical examples.

\bigskip

\begin{center} Part I.
\end{center}

\section{Ordering eigenvalues}

In all standard discussions on the distribution of eigenvalues of a matrix one begins with ordering the eigenvalues as
$$\lambda_1 \leq \cdots \leq \lambda_n$$
In case of a matrix with functions as entries the "eigenvalues"
are themselves functions and it may not be possible to arrange these eigenvalue-functions in an order. So we begin with this problem.
There are other possible ways to deduce some of our results but our proofs are most direct and constructive. They also set up the favours of our program.

We fix a $\sigma$-finite measure space $(\Omega, \mu)$. By a $n \times n$ random matrix $A =(a_{ij})$ on $\Omega$ we shall mean simply that
the entries $a_{ij} : \Omega \to \CC$ are random variables on $\Omega$. We shall not impose the iid condition in the beginning. In this sense we consider $A$ as a random linear operator $A: \Omega \times \CC^n \to \CC^n$ given by
$A(\omega, v) = (a_{ij}(\omega))v$.

We begin with stating a simple lemma in measure theory.

\begin{lem}\label{lem:mpartition}
Let $\Omega = \bigcup_{i\geq1} \Omega_i$ be a covering of $\Omega$ by at most countably  many measurable subsets.
Set
$$
g_1 := 1_{\Omega_1},
$$
and
$$
g_i := 1_{\Omega_i}(1 - g_1-\cdots - g_{i-1}), \quad i=2,3,\ldots.
$$
Then
$$
1 = \sum_{i} g_i
$$
is a measurable partition of  unity of $X$ subordinating  to the (an, arbitrary but fixed,
ordered) family $\{\Omega_1, \Omega_2, \ldots\}$.
More precisely, $g_i = 1_{\Omega_i\setminus \cup_{j<i} \Omega_j}$,
and $g_ig_j=0$ for all
$i\neq j$.
\end{lem} $\square$

\begin{lem}\label{lem:basis}
Let $u_1, u_2, \ldots, u_m$ be 
be
measurable functions from  $\Omega$ into $\mathbb{R}^n$
 such that $$\{u_1(x), u_2(x),\ldots, u_m(x)\}$$ is an orthonormal subset of
$\mathbb{R}^n$ for each $x$ in $\Omega$.  Then we can 
find measurable functions $u_{m+1}, \ldots, u_n$
 from  $\Omega$ into $\mathbb{R}^n$ such that
$\{u_1(x), \ldots, u_n(x)\}$ is an orthonormal basis of $\mathbb{R}^n$ for every $x$ in $\Omega$.
\end{lem}
\begin{proof}
Consider the standard orthonormal basis $\{\widehat{e}_1, \ldots, \widehat{e}_n\}$ consisting of constant vector-valued functions
$\widehat{e}_k(x) = e_k\in\mathbb{R}^n, \forall x\in X$.
For each choice of $n-m$ of them, $\gamma :=\{\widehat{e}_{\gamma_1}, \ldots, \widehat{e}_{\gamma_{m-n}}\}$, the subset
$X_\gamma$
of $x$ in $\Omega$ at which
$$
\{u_1(x), u_2(x),\ldots, u_m(x), \widehat{e}_{\gamma_1}(x), \ldots, \widehat{e}_{\gamma_{m-n}}(x)\}
$$
forms a basis of $\mathbb{R}^n$ is measurable.  Indeed, it is the cozero set of the measurable function
$$\det(u_1(x), u_2(x),\ldots, u_m(x), \widehat{e}_{\gamma_1}(x), \ldots, \widehat{e}_{\gamma_{m-n}}(x)).$$
By the Gram-Schmidt process, we can transform the linearly independent set to an orthonormal basis
$$
\{u_1(x), u_2(x),\ldots, u_m(x), f^\gamma_{m+1}(x), \ldots, f^\gamma_{n}(x)\}
$$
of $\mathbb{R}^n$ for each
$x$ in $\Omega_\gamma$ such that the functions $f^\gamma_{m+1}(x), \ldots, f^\gamma_{n}(x)$ are defined and
measurable on $\Omega_\gamma$.  Subordinating to the measurable covering $\Omega=\bigcup_\gamma \Omega_\gamma$, a measurable
partition of unity is given by
$$
1 = \sum_\gamma g_\gamma.
$$
Define
$$
u_{m+k}(x) := \sum_\gamma g_\gamma(x)f^\gamma_{m+k}, \quad k = 1, \ldots, n-m.
$$
We have all such $u_{m+k}$ measurable on $\Omega$, and together with $u_1,\ldots, u_m$ they form
an orthonormal family we want.
\end{proof}

\begin{lem}\label{lem:max-eigen}
Let $a_1, a_2, \ldots, a_n$ be
measurable functions from  $\Omega$ into $\mathbb{R}^n$.  Suppose that
the random matrix $A(x)=\left[ a_1(x)\ a_2(x)\ \ldots\ a_n(x)\right]$ is symmetric for each $x$ in $\Omega$.
Then the spectral radius  $r(A(x))$ and the maximal eigenvalue  $\lambda_{\max}(A(x))$
of the random matrix $A(x)$ are both measurable functions on $\Omega$.
\end{lem}
\begin{proof}
We note that all matrix norms on $\mathbb{R}^{n\times n}$ are equivalent, and they give the  spectral
radius by
$$
r(A(x)) = \limsup \|A(x)^n\|^{1/n}, \quad\forall x\in \Omega.
$$
If we use, for example, the Hilbert-Schmidt  norm $\|A\| := (\sum_{i,j} |a_{ij}|^2)^{1/2}$, then we see that
$r(A(x))$ is a measurable function on
$\Omega$.  Note also that $r(A(x))$ equals the operator norm of $A(x)$.  Consequently, the function
$$
x\mapsto r(A(x)+ r(A(x))I)
$$
is also measurable on $\Omega$ and gives rise to the maximal eigenvalue $\lambda_{\max}(A(x) + r(A(x))I)$ of
the positive matrix $A(x) + r(A(x))I$.  Since $\lambda_{\max}(A(x)) = \lambda_{\max}(A(x) + r(A(x))I) - r(A(x))$,
we obtain the measurability of $\lambda_{\max}(A(\Omega))$.
\end{proof}

\begin{lem}\label{lem:rank1}
Let  $A(x)=\left[ a_1(x)\ a_2(x)\ \ldots\ a_n(x)\right]$ be a real symmetric random matrix on $\Omega$ such that
$A(x)$ has rank at most one at every $x$.  Then we can order the eigenvalues $\lambda_1(x)\leq \cdots\leq \lambda_n(x)$
of $A(x)$ as measurable functions on $\Omega$ with corresponding measurable eigenvector functions $u_1(x), \ldots, u_n(x)$
which form an orthonormal basis of $\mathbb{R}^n$ for every $x$ in $\Omega$.
\end{lem}
\begin{proof}
Assume  $A(x)\geq 0$. i.e., positive (semi-definite), for every $x$ in $X$ first.  Let
$$
X_0 =\{x\in X: A(x)=0\}.
$$
On its complement, let
$$
X_k = \{x\in X: a_k(x)\neq 0\}, \quad k=1,\ldots, n.
$$
Clearly,  $\{\Omega_0,X_1,\ldots, \Omega_n\}$
forms a measurable covering of $\Omega$.
Let $1 = \sum_{i=0}^n g_i$ be the measurable partition subordinating to this covering as in Lemma \ref{lem:mpartition}.
If $x\in \Omega_k$ for some $k >0$ then $a_k(x)$ is an eigenvector of $A(x)$ for the maximal, and the unique positive,
eigenvalue $\lambda_n(x)$ of $A(x)$.  Define
$$
u_n(x) = g_0(x)\hat{e}_1(x) + \sum_{i=1}^n g_i \frac{a_i(x)}{\|a_i(x)\|}.
$$
Then we have a measurable function from $\Omega$ into $\mathbb{R}^n$ such that $\|u_n(x)\|=1$ everywhere on $\Omega$.
By Lemma \ref{lem:basis}, we can enlarge it to have a measurable orthonormal basis
$$
\{u_1(x), u_2(x), \ldots, u_n(x)\}
$$
on $X$.  We set $\lambda_1=\cdots=\lambda_{n-1}= 0$.
Note that $\lambda_n$ vanishes on $\Omega_0$ and takes strictly positive values elsewhere.

In the general case,
let $\Omega_+$ (resp.\ $\Omega_-$) be the subset of $\Omega$ of those $x$ at which the trace of $A(x)$ is non-negative (resp.\ non-positive).
We divide $\Omega$ into a  measurable union $\Omega= \Omega_+\cup \Omega_-$ such that
$A(x)\geq 0$ on $X_+$, and $-A(x)\geq 0$ on $\Omega_-$.
Applying above arguments separately on $\Omega_+$ and $\Omega_-$, and
gluing the results together with Lemma \ref{lem:mpartition}, we obtain the assertion.
\end{proof}

\begin{thm}\label{thm:order-eigen}
Let   $A(x)=\left[ a_1(x)\ a_2(x)\ \ldots\ a_n(x)\right]$ be a real symmetric random matrix on $\Omega$.
 Then we can order the eigenvalues $\lambda_1(x)\leq \cdots\leq \lambda_n(x)$
of $A(x)$ as measurable functions on $\Omega$ with corresponding measurable eigenvector functions $u_1(x), \ldots, u_n(x)$
which form an orthonormal basis of $\mathbb{R}^n$ for every $x$ in $\Omega$.
\end{thm}
\begin{proof}
By Lemma \ref{lem:rank1}, the assertion holds for the case $n=1$.
Assume by induction the assertion is valid for all dimension less than $n$.
Suppose also that $A(x)\geq 0$ everywhere on $\Omega$.
By Lemma \ref{lem:max-eigen}, we see that the maximal eigenvalue $\lambda_n(x)$ of $A(x)$ is a
measurable function on $\Omega$.
We are going to show that we can find an associated measurable eigenvector function $u_n(x)$ of $A(x)$ of norm one everywhere
on $\Omega$.

For each $x$ in $\Omega$, let
$$
1=P_n(x) + P_{n-1}(x) + \cdots + P_{n-m}(x)
$$
be the orthogonal sum of eigenprojections of $A(x)$
for distinct eigenvalues $\lambda_n(x)> \lambda_{n-1}(x) > \cdots > \lambda_{n-m}(x)$.
In particular, $P_n(x)$ is the (nonzero) orthogonal projection of $\mathbb{R}^n$ onto
the eigenspace of $A(x)$ for the maximal eigenvalue $\lambda_n(x)$.
Note that $m=m(x)$ depends on $x$.

Let $X_0 =\{x\in X: \lambda_n(x)=0\}$, which is a measurable set.
Let $x\in X\setminus X_0$.  For each $k=1,2,\ldots, n$,
write
$$
e_k = P_n(x)e_k + P_{n-1}(x)e_k + \cdots + P_{n-m}(x)e_k.
$$
Applying $A(x)$ on both sides repeatedly, we get
\begin{align*}
A(x)^q e_k &= \lambda_n(x)^q P_n(x)e_k + \lambda_{n-1}(x)^q P_{n-1}(x)e_k + \cdots + \lambda_{n-m}(x)^qP_{n-m}(x)e_k,
\end{align*}
for $q=1,2,\ldots$.
Consequently,
$$
\frac{A(x)^q e_k}{\lambda_n(x)^q} \longrightarrow P_n(x)e_k, \quad \text{as } q\to\infty.
$$
As a pointwise limit of measurable functions,
$P_n(x)e_k$ is measurable from  $X\setminus \Omega_0$
into $\mathbb{R}^n$.  Consequently, its cozero set
$\Omega_k = \{x\in \Omega\setminus \Omega_0: P_n(x)e_k \neq 0\}$
is measurable for $k=1,2,\ldots, n$.
Clearly, $\Omega = \Omega_0\cup \Omega_1\cup\cdots\cup \Omega_n$ is a measurable covering of $\Omega$.
Using Lemma \ref{lem:mpartition}, we have an
 associated
measurable partition of unity
$$
1=g_0 + g_1 + \cdots +g_n.
$$
We can then define a measurable function
$$
u_n(x) = \frac{g_0(x)e_1 + \sum_{k=1}^n g_k(x)P_n(x)e_k}{\|g_0(x)e_1 + \sum_{k=1}^n g_k(x)P_n(x)e_k\|},
$$
which is of norm one everywhere on $\Omega$.

By Lemma \ref{lem:basis}, we can enlarge $\{u_n(x)\}$ to a measurable orthonormal random  basis
$$
\{u_1(x), u_2(x), \ldots, u_n(x)\}
$$
of $\mathbb{R}^n$ everywhere on $\Omega$.
Define an orthogonal matrix $U = \left[u_1(x)\ u_2(x)\ \cdots\ u_n(x)\right]$.
Then
$$
U^t(x) A(x)U(x) = \left(
                    \begin{array}{cc}
                      A'(x) & 0 \\
                      0 & \lambda_n(x)\\
                    \end{array}
                  \right), \quad\forall x\in \Omega.
$$
Here, $A'(x)$ is an $(n-1)\times(n-1)$ positive semi-definite real matrix.
By the induction hypothesis, we can order its $n-1$ eigenvalues
$$
\lambda_1(x) \leq \cdots \leq \lambda_{n-1}(x), \quad\forall x\in \Omega,
$$
as measurable functions on $\Omega$.  Clearly, $\lambda_{n-1}(x)\leq \lambda_n(x)$ everywhere on $\Omega$.

In general, we consider the random matrix $B(x) = A(x)+ r(A(x))I \geq 0$.
Let the ordered measurable eigenfunctions of $B$ be $g_1, \ldots, g_n$.
Then, the eigenvalue functions of $A$ are given by
$\lambda_k(x) = g_k(x) - r(A(x))$ for $k=1,\ldots, n$.
Moreover, $A$ and $B$ share the same set of measurable eigenvector functions.
\end{proof}

We remark that with slight modification the above arguments also work for the complex case.

%We are now working on the Cauchy Interlace Theorem for random matrices.
By Theorem \ref{thm:order-eigen}, we can arrange the (all real) eigenvalues of an $n\times n$ random Hermitian matrix $A(x)$ in an increasing order
$$
\lambda_{\min}(x) = \lambda_1(x) \leq \lambda_2(x) \leq \cdots \leq \lambda_n(x) = \lambda_{\max}(x).
$$
We write $\lambda_k(A(x))$ if the random matrix $A$ is to be emphasized.

The following two lemmas are direct applications of Theorem \ref{thm:order-eigen} and the theorems of Rayleigh-Ritz,
Courant-Fischer, and Weyl.

It is now clear that the standard proofs of  the theorems of 
Rayleigh-Ritz, Courant-Fischer and Weyl
in \cite{HJ 90}, \cite{GV 96}
extend to random matrices. Together with Theorem \ref{thm:order-eigen} 
we obtain the following lemmas.

\begin{lem}\label{lem:RR-CF}
Let $A(x)$ be an $n\times n$ Hermitian random matrix on $\Omega$. Then
\begin{align*}
\lambda_{\max}(x) &= \max_{v} \frac{v(x)^*A(x)v(x)}{v(x)^*v(x)} = \max_{v(x)^*v(x)=1} v(x)^*A(x)v(x),\\
\lambda_{\min}(x) &= \min_{v} \frac{v(x)^*A(x)v(x)}{v(x)^*v(x)} = \min_{v(x)^*v(x)=1} v(x)^*A(x)v(x), \quad\forall x\in \Omega.
\end{align*}
Here, $v$ runs through all non-vanishing random vectors, i.e., measurable functions
from $\Omega$ into $\mathbb{C}^n$.  In general, for $1\leq k\leq n$ we have
\begin{align*}
\lambda_k(x)
&= \min_{w_1, \ldots, w_{n-k}}\ \max_{v(x)\bot w_1(x), \ldots, w_{n-k}(x)}
\frac{v(x)^*A(x)v(x)}{v(x)^*v(x)}\\
&= \max_{w_1, \ldots, w_{n-k}}\ \min_{v(x)\bot w_1(x), \ldots, w_{n-k}(x)}
\frac{v(x)^*A(x)v(x)}{v(x)^*v(x)}, \quad\forall x\in \Omega,
\end{align*}
where $w_1,\ldots, w_{n-k}$ and $v$ run through all non-vanishing random vectors.
\end{lem} $\square$

\begin{lem}\label{lem:Weyl}
Let $A,B$ be $n\times n$ Hermitian random matrices on $\Omega$.
Let the eigenvalue functions $\lambda_j(A(x)), \lambda_j(B(x))$ and $\lambda_j(A(x) + B(x))$ be arranged in increasing order. Then
$$
\lambda_k(A(x)) + \lambda_1(B(x)) \leq \lambda_k(A(x)+B(x)) \leq \lambda_k(A(x)) + \lambda_n(B(x)), \quad\forall x\in X.
$$
\end{lem} $\square$

With these lemmas we can see immediately that  Cauchy Interlacing theorem for random matrices also holds. This will allow us to extends the proofs of Chan in \cite{ChR 89}.

\begin{thm}\label{thm:Cauchy}
Let $\hat{A}(x)$ be an $n \times n$
Hermitian random matrix with eigenvalue functions $\{\hat{\lambda}_j(x)\}$ arranged in increasing order. Let $A(x)$ be a
principal random submatrix of $\hat{A}(x)$ of order $n-1$ with eigenvalue functions $\{\lambda_j(x)\}$ arranged in increasing
order. Then
$$
\hat{\lambda}_1(x)\leq \lambda_1(x) \leq \hat{\lambda}_2(x)\leq\lambda_2(x)\leq \cdots\leq
\lambda_{n-2}(x)\leq \hat{\lambda}_{n-1}(x)\leq\lambda_{n-1}(x)\leq\hat{\lambda}_{n}(x)
$$
for all $x$ in $\Omega$. In general, if $A(x)$ is an $r\times r$  principal random
submatrix of $\hat{A}(x)$ obtained by deleting $n-r$ rows and the corresponding columns from
$\hat{A}(x)$. Then
$$
\lambda_k(\hat{A}(x))\leq \lambda_k(A(x))\leq \lambda_{k+n-r}(\hat{A}(x)), \quad \forall x\in X,
\ 1\leq k\leq r.
$$
\end{thm} $\square$

\section{Almost equidistribution}

We extend the 
 important concept of equidistribution in analytic number theory to
 almost equidistribution.

Let
$(a^{(n)}_j)_{1\leq j \leq n+1}$,
$(b^{(n)}_j)_{1\leq j \leq n+1}$ be
two infinite lower triangular matrices
of real random variables on $\Omega$  such that for almost all $\omega$ in $\Omega$ the
entries in the lower triangular matrices are uniformly bounded, i.e.,
there is some constant $K(\omega)>0$ such that  $a^{(n)}_j(\omega)$, $b^{(n)}_j(\omega)$ are in $[-K(\omega), K(\omega)]$
for all admissible indices $n, j$.
We say that they are \emph{equally distributed} if
\begin{equation}\label{eq:heart}
\lim_{n \to \infty} \frac{1}{n+1}
\sum_{j=1}^{n+1} (f(a^{(n)}_j(\omega)) - f(b^{(n)}_j(\omega)))=0, \text{ a.e. on } \Omega,
\end{equation}
for any continuous function $f$ on $\mathbb{R}$.

\begin{prop}
The following are  equivalent conditions for a.e. uniformly bounded real infinite
random lower triangular matrices $(a^{(n)}_j)_{1\leq j \leq n+1}$ and
$(b^{(n)}_j)_{1\leq j \leq n+1}$ to be equally distributed.

(a) (\ref{eq:heart}) holds for $f(t)=t^k$ for all $k=0,1,2,\ldots$.

(b) (\ref{eq:heart}) holds for $f(t)=\log(1+x_nt)$ for a real sequence $x_n$ converging to zero.
\end{prop}

\begin{proof}
(a) Forgetting a set of measure zero,
we can assume both the random matrices are uniformly bounded everywhere on $\Omega$, and
(\ref{eq:heart}) holds for all $f(t)=t^k$ everywhere on $\Omega$ for all $k=0,1,2,\ldots$.
By the Stone-Weierstrass theorem, for every real continuous function
$f$ in $C(\mathbb{R})$ we have a sequence $\{p_m(f)\}_m$ of polynomial such that
$$
\sup \{|f(x) - p_m(f)(x)|: x\in [-m,m]\} < 1/m, \quad m =1,2,\ldots.
$$
For every fixed $\omega$ in $\Omega$, we have
\begin{align*}
&\ \lim_{n\to\infty} \frac{1}{n+1}\left|\sum_{j=1}^{n+1} (f(a^{(n)}_j(w)) - f(b^{(n)}_j(w)))\right| \\
\leq &\ \lim_{n\to\infty} \frac{1}{n+1}\left|\sum_{j=1}^{n+1} (p_m(f)(a^{(n)}_j(w)) - p_m(f)(b^{(n)}_j(w)))\right| + 2/m\\
= &\ 2/m, \quad\forall m\geq K(\omega).
\end{align*}
This gives rise to
$$
\lim_{n\to\infty} \frac{1}{n+1}\sum_{j=1}^{n+1} (f(a^{(n)}_j(\omega)) - f(b^{(n)}_j(\omega))) =0, \quad\forall \omega\in \Omega.
$$

(b) For (b), we recall that in \cite{GS 58} the assumption should be that
(\ref{eq:heart}) holds for $f_x(t)=\log(1+xt)$ for all real $x$ and $|x| < K^{-1}$.
Indeed, to utilize the proof there, i.e. to use Vitali's theorem for holomorphic
functions  (\cite{Tit 60} p.168)
 %\footnote{See, e.g., \tt{http://mathworld.wolfram.com/VitalisConvergenceTheorem.html}.},
 we simply need
(\ref{eq:heart}) to hold for a sequence of real $x_n$ with $|x_n| < K^{-1}$ and $\{x_n\}$ has a cluster point $x$ with
$|x| < K^{-1}$.
In this case, the proof in page 63 of \cite{GS 58}
works.

For the sake of completeness, we include the  proof here.
Assume that for each occasion $z=x_m$, where  $x_m$ is  a sequence
of real numbers with $|x_m|\leq R< K^{-1}$
and having a cluster point, we have
$$
\lim_{n\to\infty} \frac{\sum_{j=1}^{n+1} (\log(1+za^{(n)}_j(\omega)) -
\log(1+zb^{(n)}_j(\omega)))}{n+1} =0
$$
holds
all $\omega$ in  $\Omega$ except a measurable subset of measure zero.
Fix $\omega$, the quotient following the limit sign in the above displayed
formula is a single-valued and analytic function of the complex
variable $z$ provided that $|z|< 1/K$.  It is uniformly bounded
in $z$ and $n$ provided that $|z| \leq R_1$ with $R< R_1 < 1/K$.
Since the sequence converges
on a set $\{x_1, x_2, \ldots\}$ with cluster point in the
complex open disk
$B(0;R_1)$
centered at zero with radius $R_1$, Vitali's theorem ensures that
the sequence indeed converges uniformly on $\overline{B(0;R)}$.
Applying the Cauchy integral formula, we see that
$$
\lim_{n\to\infty} \frac{\sum_{j=1}^{n+1} (a^{(n)}_j(\omega)^m -
b^{(n)}_j(\omega)^m)}{n+1}, \quad\forall m=0,1,2,\ldots.
$$
Then we can apply (a). \end{proof}

A useful 
special case of the previous proposition is when $K$ is constant. 
The obvious statements are left to the readers.

\bigskip

%===============================
\begin{center} Part II.
\end{center}

\bigskip

Given a sequence $\{ c_n: -\infty < n < \infty \}$ of complex numbers. We can form an $(n+1) \times (n+1)$ matrix
$T_n$ whose
$(i,j)$-th entry is $c_{i-j}$.

On the other hand a random Toeplitz matrix is like the $T_n$ given above except now $c_n$ is a sequence of independent identically distributed (iid) random variables with Gaussian distributions. See
\cite{BDJ 06}.

Let us at the moment ignore the iid condition.
Suppose we are given a sequence $\{ c_n: -\infty < n < \infty \}$ of real valued random variables on a probability space
$(\Omega, \mu)$. Suppose there is a function $f(\omega, x)$ on
$\Omega  \times S^1$ such that it has a  \emph{Fourier series expansion} at almost everywhere on $\Omega$ that
$$
f(\omega, z) = \sum_{n=-\infty}^{\infty} c_n(\omega) e^{inz}, \quad \forall z\in S^1.
$$
We say that the above Fourier series is \emph{uniformly summable} on $\Omega$ if for every $\epsilon >0$ there
is a positive integer $N$ such that outside a subset of $\Omega$ of zero measure we have
$$
 \sum_{|n| > N} |c_n(\omega)| < \epsilon.
$$

We shall establish   the results  of Szeg\'o
necessary for the analysis of
 Chan in  \cite{ChS 89} and  \cite{ChR 89}.

\section{Random Laurent operators}

Let $(\Omega, \mu)$ be a measure space, $H$ be a Hilbert space and
$\sL(H)$ be the $C^*$ algebra of bounded linear operators on $H$.
By a random linear operator we mean a map
$\Phi: \Omega \to \sL(H)$ such that for every pair $x,y \in H$ the map
$$\Omega \to \CC : \omega \mapsto \la \Phi (\omega) x, y\ra$$
is measurable. See \cite{Sko 84}.

We begin with a construction of random Laurent  operators
following Brown and Halmos \cite{{BH 63}}.

Let $(\Omega, \mu)$ be a probability space. Let $S^1$ be the unit circle in the complex plane. $\fL^2(S^1)$ denotes the Hilbert space of square integrable
complex valued
functions on $S^1$ with respect to the Lebesgue measure on the Borel sets.

The space of bounded linear
operators on a Hilbert space $H$ is denoted by $\sL (H)$.
But
we shall write $\sL(S^1)$  for  $\sL(\fL^2(S^1))$.

Let $\varphi: \Omega  \times S^1 \to \CC$ be measurable
with respect to the product measure. Then the function
$\varphi(\omega, \bullet)$ defines a measurable function
$\varphi_\omega: S^1 \to \CC$ which we shall assume to be bounded.
 By the Laurent operator defined by
$\varphi$ we mean the map $L=L^\varphi : \Omega \to \sL(S^1)$ defined by multiplication
$$L_\omega f = \varphi_\omega f, \;\; f \in \fL^2(S^1), \omega \in \Omega.$$

\begin{prop}
Given $f, g \in \fL^2(S^1)$ the map $\Omega \to \CC$ taking $\omega$
to $\la  L_\omega f, g\ra$ is measurable.
\end{prop}
\begin{proof}
Note that every measurable function can be written as a pointwise limit of finite linear sums
of simple functions.  We can use simple functions of the form as the indicator functions
$1_{A\times B}(\omega, z) = 1_A(\omega)1_B(z)$ of measurable squares $A\times B$.
Such functions give rise the measurability of the map,
and so do the finite linear sums of them.  Taking pointwise limit of a
sequence we verify the asserted
measurability.
\end{proof}

\begin{lem}
Suppose $A: \Omega \to \sL(S^1)$ is a random linear operator. Then the functions $\varphi : \Omega \times S^1 \to \CC$ defined by
$\varphi(\omega, z) := (A(\omega) e_0)(z)$
and $\varphi_\omega :  S^1 \to \CC$ defined by
$\varphi_\omega (z) = \varphi(\omega, z) $, for $\omega \in \Omega$
are measurable.
\end{lem}
\begin{proof}
Let $\{e_1, \ldots\}$ be an orthonormal basis of the separable Hilbert space $\fL^2(S^1)$.
The measurability of $A$ implies that the maps
$$
(\omega,z) \mapsto \sum_{j=1}^n \la A(\omega)e_0, e_j\ra e_j(z)
$$
are measurable on $\Omega\times S^1$ for $n=1,2,\ldots$.
Letting $n\to\infty$, we see that the map
\begin{align*}
\varphi(\omega, z) &= (A(\omega) e_0)(z)=\sum_{j=1}^\infty \la A(\omega)e_0, e_j\ra e_j(z)
\end{align*}
is measurable on  $\Omega\times S^1$.  By the definition of product measures, we see
that $\varphi_\omega$ is measurable on $S^1$ for all $\omega$ in $\Omega$.
\end{proof}

\begin{lem}
Let $\phi$ be a  measurable function in $\fL^2(S^1)$. Define a linear map $B: \fL^2(S^1) \to \fL^2(S^1)$ by
$f \mapsto \phi f$. Then the domain of $B$ is dense and  $B$ is a closed operator.
\end{lem}
\begin{proof}
Since $\phi g\in \fL^2(S^1)$ for all $g$ in $C(S^1)$, the domain $D(B)$ of $B$ contains the dense
subspace $C(S^1)$ of $\fL^2(S^1)$.  Suppose that $f_n\to f$ and $\phi f_n\to g$ in $\fL^2$-norm.
Then both $f_n\to f$ and $\phi f_n \to g$ in measure.  Therefore, we have a subsequence $\{f_{n_k}\}$ of $\{f_n\}$
such that both $f_{k_k}\to f$ and $\phi f_{n_k}\to g$ almost everywhere on $S^1$.  Consequently, $\phi f =g$ and $B$ has
a closed graph.
\end{proof}

For an integer $n$ and  $z \in S^1$, we write $e_n (z)=z^n$.
We introduce the shift operator $W$ on $\fL^2(S^1)$ as multiplication by $e_1$, i.e.
$W f (z) = (e_1 f)(z)= z f(z)$.
Clearly $W e_n = e_{n+1}$ for all $n$ and $W^n f = e_n f$ for
all $ n  \geq 0$. We are interested in the centralizer $Z(W)$ of
the shift operator $W$
in the space $\sR$ of all random linear operators from
$\Omega$ to $\sL(S^1)$-
$$Z(W)= \{A \in \sR : A(\omega) W = W A(\omega), \;\omega \in \Omega\}$$.

\begin{thm}
We have $Z(W) =\{L^\varphi\}$ where $\varphi : \Omega \times S^1 \to \CC$ with $\varphi_\omega$ bounded for $\omega \in \Omega$.
\end{thm}
\begin{proof}
Since multiplication operators commutes it follows that
$\{L^\varphi\} \subset  Z(W)$.

 Conversely take $A \in Z(W)$. Define   $\varphi : \Omega \times S^1 \to \CC$ defined by
$\varphi(\omega, z) := (A(\omega) e_0)(z)$
Then for $n \geq 0$, we have
$$A(\omega)e_n = A(\omega)W^n e_0 = W^n A (\omega)e_0 = W^n \varphi_\omega = e_n \cdot
\varphi_\omega = \varphi_\omega \cdot e_n$$

 Let $B(\omega)$ be the operator defined by  multiplication by
 $\varphi_\omega$. Then the above equations that $A(\omega)e_n =
 B(\omega) e_n$ for  $n \geq 0$. Fuglede's theorem
 (\cite{Fug 50}) says that
 $A(\omega)$ commutes with $W$ implies that $A(\omega)$
 commutes with $W^*$. The same argument that gives $A(\omega)e_n =
 B(\omega) e_n$ for all $n$. Thus $A(\omega) =
 B(\omega) $ on $\fL^2(S^1)$.

 That $\varphi_\omega$ is bounded follow by a norm argument.
 \end{proof}

\begin{prop}
Let $\Omega$ be a measure space.  Let $S^1$ be the unit circle in the complex plane $\mathbb C$.
Let $f : \Omega\times S^1 \to \mathbb R$  (or  $\mathbb C$) such that
\begin{itemize}
    \item for each fixed $\omega$ in $\Omega$, the map $z\mapsto
     f(\omega, z)$ is continuous, and
    \item for each fixed $z$ in $S^1$, the map 
    $\omega\mapsto f(\omega,z)$ is measurable.
\end{itemize}
Then there are measurable functions $g_n(\omega)$ on $\Omega$ such that
$$
f(\omega,z) = \sum_{n\in\mathbb Z} g_n(\omega)z^n, \quad\forall \omega\in \Omega, \forall z\in \mathbb S^1.
$$
\end{prop}
\begin{proof}
Since $C(S^1)$ is a separable Banach space, its  weak* compact convex dual ball is
metrizable.  In particular, the norm one linear functional
$$
h(z) \mapsto \frac{1}{{2\pi}}\int_{S^1} h(z) z^{-n} dz, \quad h \in C(S^1),
$$
is a limit of a sequence of convex combinations of point masses.

For a point mass $\delta_t$ with $t$ in $S^1$, the function
$$
\omega \mapsto \delta_t(f(\omega,z)) = f(\omega,t)
$$
is measurable on $\Omega$.  As a pointwise limit of a sequence of measurable functions, the function
$g_n(x)$ defined by
$$
\omega \mapsto  \frac{1}{{2\pi}}\int_{S^1} f(\omega,z) z^{-n} dz
$$
is measurable on $\Omega$ for each $n$ in $\mathbb Z$.

It is plain that $f$ carries the stated form (by Fourier transform).
\end{proof}

The matrix coefficients of  a random linear operator $A: \Omega \to \sL(S^1)$ with respect to $\{e_j\}$ are defined to be
$$a_{ij} (\omega) = \la A(\omega) e_j, e_i\ra$$
for $i. j \in \ZZ$.

\begin{lem} Given a measurable function
$\varphi : \Omega \times S^1 \to \CC$ with $\varphi_\omega$ bounded.
Let $\varphi_\omega = \sum_i c_i (\omega)e_i$ be the Fourier expansion of  $\varphi$.
Let $L=L^\varphi$ be the random linear operator defined by multiplication by $\varphi$.  Then the matrix coefficients of $L$ are given by
$$\ell_{ij}(\omega) = c_{i-j}(\omega)$$
\end{lem}

\begin{thm}
A random linear operator $A: \Omega \to \sL(S^1)$ is a Laurent random operator if and only if its matrix coefficients satisfy
$$a_{i+1,j+1} (\omega) = a_{ij} (\omega)$$
for $\omega \in \Omega$.
\end{thm}
\begin{proof}
After the preceding lemma it remains to show that the condition is sufficient which will follow if we show that $A \in Z(W)$.
But $a_{i+1,j+1} (\omega) = a_{ij} (\omega)$ implies
$$\la A(\omega) W e_j, e_i \ra =  \la WA(\omega) e_j, e_i \ra. $$
\end{proof}

\section{Random Toeplitz operators}

Let
$\fH^2(S^1)$ be the space of square integrable analytic functions on $S^1$. Write $P: \fL^2(S^1) \to \fH^2(S^1)$
for the projection.

Let $\varphi: \Omega  \times S^1 \to \CC$ be measurable
with respect to the product measure. Write
$\varphi_\omega: S^1 \to \RR$ for the function
$\varphi( \omega, \bullet)$. We assume that $\varphi_\omega$ are bounded.
The random Toeplitz operator defined by $\varphi$ is the map
$T=T^\varphi:  \Omega \times \fH^2(S^1)  \to \fH^2(S^1)$  to be given  by
$$T_\omega (u) = P(\varphi_\omega \cdot u), \;\;
u \in \fH^2(S^1).$$
That is $T^\varphi f = PL^\varphi f$ for $ f \in \fH^2(S^1)$.

We can compute the matrix coefficients of $T^\varphi$
 with respect to $\{e_j : j \geq 0\}$ as follows
 \begin{align*} \la T^\varphi (\omega)e_j, e_i \ra &=
 \la PL^\varphi (\omega) e_j, e_i \ra  =\la L^\varphi (\omega) e_j, e_i \ra= \la L^\varphi (\omega) e_{j+1}, e_{i+1}\ra \\
 &= \la PL^\varphi (\omega) e_{j+1}, e_{i+1}\ra
 =\la T^\varphi (\omega)e_{j+1}, e_{i+1} \ra
\end{align*}
Thus we see that the matrix coefficients
$t_{ij}(\omega) = \la T^\varphi (\omega)e_j, e_i \ra $
of
random Toeplitz operator defined by $\varphi$
satisfy
$$t_{i+1, j+1}(\omega)= t_{ij}(\omega).$$

\begin{lem}
Let $A_n: \Omega \to \sL(H)$ be a sequence of random linear operators.
Suppose that for each $\omega \in \Omega$ the sequence of operators
$A_n(\omega)$ has a weak limit $A_\infty^\omega$. Put
 $A_\infty(\omega) = A_\infty^\omega$. Then
 $A_\infty: \Omega \to \sL(H)$ is a  random linear operator.
\end{lem}
\begin{proof}
We put on  $\sL(H)$
 the weak operator topology, which
is defined by the seminorms
$$
T \mapsto \langle Tx,y\rangle
$$
where $x,y$ are vectors in $H$.
In this case, for every $\omega$ in $\Omega$, and $x,y$ in $H$ we have
$$
\langle A_n(\omega) x, y\rangle \longrightarrow \langle A_\infty(\omega)x, y\rangle, \text{ as } n\to \infty.
$$
This says exactly the map $\omega\mapsto \langle A_\infty(\omega)x,y\rangle$ is the pointwise limit of a
sequence of measurable functions.  So it is measurable.
\end{proof}

\begin{thm}
A random linear operator $A: \Omega \to \sL(\fH^2(S^1))$ is a  random Toeplitz operator if and only if its matrix coefficients satisfy
$$a_{i+1,j+1} (\omega) = a_{ij} (\omega)$$
for $\omega \in \Omega$ and $i, j \geq 0$.
\end{thm}

Given a sequence $\{ c_n : -\infty < n < \infty \}$ of complex valued random variables on $\Omega$.  We can form an infinite matrix $T$ whose
$(i,j)$-th entry is $t_{ij}= c_{i-j}$. We can write down the first
$(n+1) \times (n+1)$ submatrix as
$$T_n = \begin{pmatrix}
c_{0} & c_{-1} & c_{-2} &\dots  & c_{-n} \\
c_{1} & c_{0} &c_{-1} &\dots  & c_{-(n-1)} \\
c_{2} &c_{1} & c_{0} &\dots &c_{-(n-2)}\\
\hdotsfor{5}\\
c_{n-2} & c_{n-3}& c_{n-4} &\dots  & c_{-2}\\
c_{n-1} & c_{n-2}& c_{n-3} &\dots  & c_{-1}\\
c_{n} & c_{n-1}& c_{n-2} &\dots  & c_{0}
\end{pmatrix}
$$
%We shall write $T_n(\omega)$ for $(c_{i-j}(\omega))$
We shall always assume that $T_n$ is hermitian.

%\section{Minimal value}

%\section{Grenander-Szeg\"o}

Let $(\Omega, \mu)$ be a probability space.
Let $f(\omega, x)$ such that $f(\omega, \bullet)$  is in $L^1(S^1)$
for a.e. in $\omega \in \Omega$. We can consider its Fourier coefficient
$$c_n(\omega) = \frac{1}{2 \pi}\int_{-\pi}^{\pi}
e^{-inx}f(\omega, x)dx.$$
We use these coefficients to build the finite Toeplitz forms
\begin{align*}
T_n(f)(\omega)&=\sum_{0 \leq j, k \leq n} c_{k-j} (\omega) u_j \bar{u}_k\\
&=\frac{1}{2 \pi}\int_{-\pi}^{\pi}
|u_0 + u_1 e^{ix}+ \cdots+ u_ne^{inx}|^2
 f(\omega, x)dx
\end{align*}
- this is a hermitian form in the variables $u$.

The eigenvariables of the Hermitian form
$T_n(f)(\omega)$ are defined as the roots of the characteristic equation $\det (T_n(f) - \lambda)=0$; we denote them by
$$\lambda^{(n)}_1(\omega), \ldots, \lambda^{(n)}_{n+1}(\omega)$$
These are real valued random variables defined a.e. in $\Omega$.

If we assume that (1) $m \leq f(x, w) \leq M$ for all $x \in S^1$ and a.e. $\omega$ and that
(2)
$$I = \sum_{p=0}^n |u_p|^2=
\frac{1}{2 \pi}\int_{-\pi}^{\pi}
|u_0 + u_1 e^{ix}+ \cdots+ u_ne^{inx}|^2
 dx =1
$$
Then it follows from the definition of
$T_n(f)(\omega)$ that
$$m \leq T_n(f)(\omega) \leq M.$$

To each eigenvariable $\lambda^{(n)}_j(\omega)$ we have a nonvanishing eigenvector
$u=(u_{j0}, \ldots, u_{jn})$ determined up to a scalar, such that
$$T_n(f)(\omega) = \lambda^{(n)}_j(\omega)\sum_{p=0}^n |u_p|^2 = \lambda^{(n)}_j(\omega)$$
for our choice of $u$. Thus we see that
$$m \leq \lambda^{(n)}_j(\omega) \leq M.$$

\begin{thm}
The matrices $(\lambda^{(n)}_j(\omega) )$ and
$(f(\omega, -\pi + \frac{2j\pi}{n+2}))$ are equally distributed.
\end{thm}
\begin{proof}
If we take $F(t) = \log t$, $t>0$, and make use of a Riemann sum, then we need to show
$$\lim_{n \to \infty} \frac{1}{n+1} \sum_{j =1}^{n+1}
\log \lambda^{(n)}_j(\omega)  = \frac{1}{2 \pi}
\int_{-\pi}^{\pi} \log f(\omega, x) dx$$

Suppose $n >0$. Let $D_n(f)$ be the determinant of $T_n(f)$. Then
$$D_n(f)(\omega) = \lambda^{(n)}_1(\omega) \cdots \lambda^{(n)}_{n+1}(\omega).$$

So we need to show
$$\lim_{n \to \infty} (D_n(f)(\omega))^{ \frac{1}{n+1}}=
\exp (\frac{1}{2 \pi}
\int_{-\pi}^{\pi} \log f(\omega, x) dx).
$$
But this follows from the minimal value theorem
$$\lim_{n \to \infty} \frac{D_n(f)(\omega)}{D_{n-1}(f)(\omega)} =
\exp (\frac{1}{2 \pi}
\int_{-\pi}^{\pi} \log f(\omega, x) dx).$$
The proof of the minimal value theorem for orthogonal polynomials given in  \cite[\S 2.2]{GS 58} extends easily to our situation.
\end{proof}

\bigskip

\begin{center} Part III.
\end{center}

\section{Circulant preconditioner}

We want to solve efficiently a linear system of the form
$T_n \mathbf{x} = \mathbf{b}$. The method of pre-conditioning means that we find a matrix $S_n$ such that
it is efficient to solve the linear system
$S_n^{-1}T_n \mathbf{x} = S_n^{-1}\mathbf{b}$.
The conjugate gradient method used to solve the above linear system says that the more the eigenvalues
of the coefficient matrix  $S_n^{-1}T_n$
are clustered together the faster the convergence rate.
 A sequence of matrices $\{A_n\}_1^\infty$ is
said to have clustered spectra around $1$ if for any
given $\epsilon > 0$ there exist positive integers
$n_1$ and $n_2$ such that for all $n > n_1$, at most
$n_2$ eigenvalues of the matrix $A_n -I_n$ have absolute value larger than $\epsilon$.

This leads to the problem of the distribution of the eigenvalues of the matrix $S_n^{-1}T_n$ when $S_n$ is the Strang circulant in the paper \cite{ChR 89} of R. Chan.
The basic technique is due to Szeg\"o which assumes that
there is a function $f$ on the circle $S^1$ such that the given sequence $\{ c_n\}$ is the sequence of Fourier coefficients of $f$ and then the method of orthogonal polynomials is applied. The function $f$ is then called the generating function of the sequence $\{ c_n\}$.

Let $(\Omega, \mu)$ be a probability space.
Let $f(\omega, x)$ be a real-valued function on
$\Omega \times S^1 $ such that $f(\omega, \bullet)$  is in $L^1(S^1)$ for a.e. in $\omega \in \Omega$. We can consider its Fourier coefficient
$$c_n(\omega) = \frac{1}{2 \pi}\int_{-\pi}^{\pi}
e^{-inx}f(\omega, x)dx.$$

We use these coefficients to build the $(n+1) \times (n+1)$ Toeplitz
matrix $T_n$ whose
 $(k,j)$ entry $(T_n)_{k,j}$  is $c_{k-j}$.
We have written $c_{k-j}$ for the function $c_{k-j}(w)$.
We shall consider the case that he $T_n$ are Hermitian positive definite a.e. in $\omega$.

The Strang preconditioner $S_n =(s_{k,j})$ of $T_n$ is the Hermitian circulant  defined as follows in two cases.

If $n=2m+1$ then $s_{k,j} = s_{k-j}$ are given by
$$ s_\ell =
 \begin{cases} c_\ell &\text{for} \; 0 \leq \ell \leq m, \\
 c_{\ell -n} &\text{for} \; m < \ell \leq n-1 \\
 \bar{s}_{-\ell} &\text{for} \; 0 < -\ell \leq n-1.
\end{cases}$$

If $n=2m$ then $s_{k,j} = s_{k-j}$ are given by
$$ s_\ell =  \begin{cases}
c_\ell &\text{for} \; 0 \leq \ell \leq m -1, \\
0 \;\text{or}\; \frac{1}{2}(c_m + c_{-m}) &\text{for} \; \ell = m,\\
c_{\ell -n} &\text{for} \; m < \ell \leq n-1 \\
\bar{s}_{-\ell} &\text{for} \; 0 < -\ell \leq n-1.
\end{cases}$$

\begin{thm} Write the sup-norm as
$\Vert c_n \Vert_\infty = \sup_{w \in \Omega} |c_n(\omega)|$.
Suppose $f$ is positive and
$\sum_n \Vert c_n \Vert_\infty$ is finite (Wiener class). Then for all
$\epsilon > 0$, there exist $M$ and $N > 0$ such that for all $n > N$, at most
$M$ eigenvalue functions of $S_n - T_n$ have sup norms exceeding
$\epsilon$.
\end{thm}

This is theorem 2 of \cite{ChR 89}; see also theorem 4 \S 8 of
\cite{ChS 89}. The proof of Chan works for we have Cauchy Interlace theorem
in our case.

\section{Numerical results}

To test our preconditioners we shall compute using random phases generated by random sequences in the following manner.
Let $\NN$ be the set of all integers $\geq 0$. The generating function will be $f : \NN \times S^1 \to \CC$ with Fourier series
$$f = \sum_{k= - \infty}^\infty c_k(t) e^{-ik \theta}. $$
We shall take a simple example of a random variable by choosing $c_k(t) = a_k  e^{i \phi_k (t)}$.
For $a_k$ we shall take the example as given in \cite{ChR 89} namely
$$ a_k =  \begin{cases} \frac{1 + \sqrt{-1}}{(1 +k)^{1.1}}, \;\; & k > 0, \\
2 \;\; & k =0\\
\bar{a}_{-k},  \;\; & k< 0.
\end{cases}
$$
For each $k > 0$
we take random sequences $\{\phi_k (t): t = 0,1,2, \ldots \}$.
Note that $f$ remains in the Wiener class. The $c_k(t)$ are used to build the Toeplitz matrices $T_n(t)$.

We choose the circulant preconditioner $S_n$ as above and run the PCG for $T_n x = b$
as given in \cite{ChS 89} p. 106, namely, start from $x_0 =0$ and $r_0 =b$. Solve
\begin{align*}
Sz_{j-1} &= r_{j-1} \\
\beta_j &= z^T_{j-1} r_{j-1}/z^T_{j-2}r_{j-2} \\
d_j &= z_{j-1} + \beta_j d_{j-1} \\
\alpha_{j-1} &= z^T_{j-1} r_{j-1}/ d^T_j T d_j \\
x_j &= x_{j-1} + \alpha_j d_j \\
r_j &= r_{j-1} - \alpha_j T d_j.
\end{align*}

 We input random sequences $\phi_k(t)$ which are Gaussian $N(0,1)$.
 Write $r_j$ for the residue after $j$ iterations.
 We calculate
 the number $I(n, t)$ of iterations required in order that Strang's circulant PCG applied to
 $T_n(t)$ achieve the residue ratios
 $$\frac{\| r_j \|_2}{\| r_j \|_0} < {10}^{-10}.$$

\begin{figure}[H]
\begin{tikzpicture}[y=.04cm, x=.06cm,font=\sffamily]
\begin{axis}[xlabel=$t$, ylabel=iterations,legend style={legend pos=north east}]
	\addplot [color=blue,mark=*,mark size=1pt,dotted] table[x=try,y=iter] {A.data};
    \addlegendentry{$T_n$}
	\addplot [color=red,mark=*,mark size=1pt] table[x=try,y=iter-p] {A.data};
    \addlegendentry{$S_n^{-1}T_n$}
\end{axis}
\end{tikzpicture}
\caption{\small Number of iterations for $n=65$, $1\leq t\leq 100$.}\label{fig1}
\end{figure}
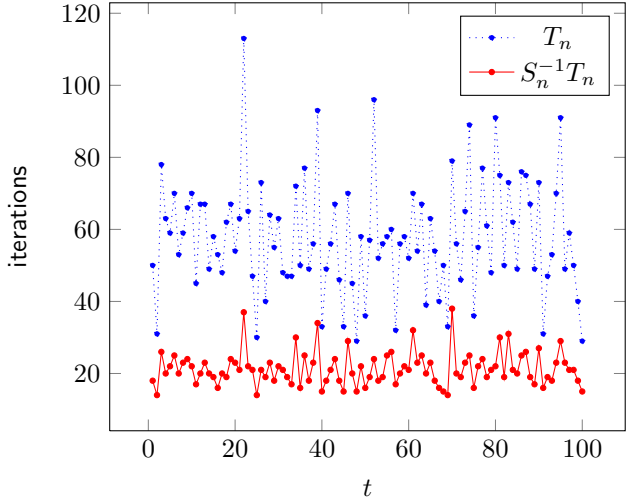

In Figure~\ref{fig1} we display the results for $n =65$ i.e. a $65 \times 65$ Toeplitz matrix $T_{65}(t)$.
The $X$-axis shows the value of $t$ and the $Y$-axis shows the number of iterations. The graph above is for
the usual conjugate gradient method for
$T_n$ and the graph below is that of $I(65, t)$
for the preconditioned $S_N^{-1}T_n$ which shows a mean value of around $20$ iterations.

\begin{figure}[H]
\begin{tikzpicture}[y=.04cm, x=.06cm,font=\sffamily]
\begin{axis}[xlabel=$n$, ylabel=iterations,legend style={legend pos=north west}]
	\addplot [color=blue,mark=*,mark size=1pt,dotted] table[x=n,y=iter] {B.data};
    \addlegendentry{$T_n$}
	\addplot [color=red,mark=*,mark size=1pt] table[x=n,y=iter-p] {B.data};
    \addlegendentry{$S_n^{-1}T_n$}
\end{axis}
\end{tikzpicture}
\caption{\small Average number of iterations with $1 \leq t \leq 50$  for each $n$ where $n=2m+1$, $10\leq m\leq 120$.}\label{fig2}
\end{figure}
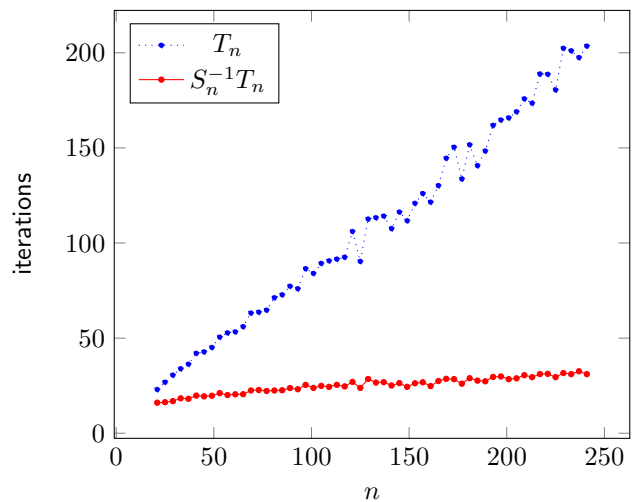

Figure~\ref{fig2} shows the variation of the average number of required iterations with the size of the Toeplitz matrices.
Here the $X$ axis is $n$ to indicate that the Toeplitz matrix is $n \times n$. We take $n$ to be odd $n = 2m+1$ and we run for $10  \leq m \leq 120$. The $Y$ axis gives the average  number of required iterations. The lower graph shows
the average
$\frac{1}{50} \sum_{t=1}^{50} I (n, t)$

\begin{figure}[H]
\begin{tikzpicture}[y=4cm, x=3cm, z=0.5cm]
\begin{axis}[xlabel=$x$, ylabel=$y$, zlabel=$t$]
	\addplot3+[only marks,mark size=1pt,scatter] table {sheets.data};
\end{axis}
\end{tikzpicture}
\caption{\small Eigenvalues.}\label{fig3}
\end{figure}
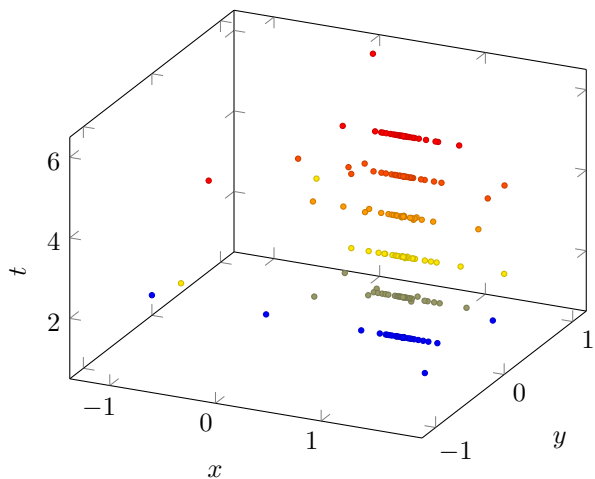

Just to confirm the clustering of the eigenvalues we show the eigenvalues of
$S_n^{-1}T_n$ in Figure~\ref{fig3}. The eigenvalues $x + \sqrt{-1}y$ are displayed as
$(x,y)$ for $t = 1,2, \ldots,6$.


\begin{thebibliography}{abc 999}

\bibitem[AGZ 09]{AGZ 09} G. Anderson, A. Guionnet, O. Zeitouni, \emph{An introduction to random matrices}, Cambridge University Press, (2009).

 \bibitem[Axe 85]{Axe 85} O. Axelsson, \emph{A survey of preconditioned iterative
 methods for linear systems of algebraic equations},
    BIT Numerical Mathematics
\textbf{25}, (1985) 165-187.

\bibitem[Axe 94]{Axe 94} O. Axelsson, \emph{Iterative solution methods}, Cambridge University Press, (1994).



\bibitem[BCC 09]{BCC 09} Z.  Bai, Y. Chen, Y. Liang, \emph{Random matrix theory and its applications},
World Scientific Pub. Co. (2009).

\bibitem[BS 10]{BS 10} Z. Bai, J. Silverstein,  \emph{Spectral analysis of large dimensional random matrices}, Springer Verlag, (2010).

\bibitem[BDJ 06]{BDJ 06} W. Bryc, A. Dembo, T. Jiang, \emph{Spectral measure of large random Hankel, Markov and Toeplitz matrices}, Annals of Prob. \textbf{34}, (2006) 1-38.


\bibitem[BH 63]{BH 63} A. Brown, P. Halmos,  \emph{Algebraic properties of Toeplitz operators}, J. Reine Angew. Math. \textbf{213}, (1963) 89-102.

\bibitem[ChS 89]{ChS 89} R. Chan, G. Strang, \emph{Toeplitz equations by conjugate gradients with circulanr preconditioner}, 
SIAM J. Sci. Stat. Comput., \textbf{10}, (1989) 104-119.

\bibitem[ChR 89]{ChR 89} R. Chan, \emph{Circulant preconditioners for Hermitian Toeplitz  systems}, SIAM J. Matrix Anal. Appls., \textbf{10}, (1989)  542-550.



\bibitem[ChR 91]{ChR 91} R. Chan, \emph{Toeplitz preconditioners for Toeplitz systems
with nonnegative generating functions}, IMA
J. Numer. Anal.  \textbf{11}, (1991) 333-345.

\bibitem[CJ 07]{CJ 07} R. Chan,  X. Jin, \emph{An introduction to iterative Toeplitz solvers}, SIAM, (2007).

\bibitem[Dan 70]{Dan 70} J. Daniel, \emph{The approximate minimization of functionals}, Prentice Hall, (1970).



\bibitem[Fug 50]{Fug 50} B. Fuglede, \emph{A commutativity theorem for normal operators}, Proc. N. A. S. ,\textbf{36}, (1950) 35-40.

\bibitem[GV 96]{GV 96} G. Golub, C. Van Loan, \emph{Matrix computations}, Johns Hopkins University Press, (1996).

\bibitem[GS 58]{GS 58} U. Grenander, G. Szeg\"o,  \emph{Toeplitz forms and their applications}, University of California Press, (1958).

\bibitem[HJ 90]{HJ 90} R. Horn, C. Johnson, \emph{Matrix analysis}, 
Cambridge University Press, (1990).


\bibitem[Meh 04]{Meh 04} M. Mehta, \emph{Random matrices}, North Holland Pub., (2004).


\bibitem[Pan 04]{Pan 04} I. Panayotov, \emph{Conjugate gradient in Hilbert spaces}, McGill University Master thesis, (2004).

\bibitem[PS 11]{PS 11} L. Pastur, M. Shcherbina \emph{Eigenvalue distribution of large random matrices}, Amer Math Society, (2011).

\bibitem[Phi 03]{Phi 03} G. Phillips, \emph{Interpolation and approximation by polynomials}, Springer Verlag, (2003).


\bibitem[Sko 84]{Sko 84} A. Skorohod, \emph{Random linear operators}, Kluwer Academic Publishers, (1984).

\bibitem[Str 86]{Str 86} G. Strang, \emph{A proposal for Toeplitz matrix calculations}, Stud. Appl. Math., 74 (1986), pp. 171-176.

\bibitem[Tao 12]{Tao 12} T. Tao, \emph{Topics in random matrix theory}, Amer Math Society, (2012).

\bibitem[Tit 60]{Tit 60} E. Titchmarsh, \emph{The Theory of Functions},  Oxford University Press, (1960).

\bibitem[Wig 55]{Wig 55} E. Wigner, \emph{Characteristic vectors of bordered matrices with infinite dimensions},
Annals of Math
\textbf{62}, (1955) 548-564.


%\bibitem[]{} , \emph{}, \textbf{}, ().

\end{thebibliography}
\end{document}